\documentclass[draft]{amsart}
\usepackage{amssymb}
\usepackage{url}
\usepackage{amscd}
\usepackage{graphicx}

 \theoremstyle{plain}
 \newtheorem{theorem}{Theorem}
 
 \newtheorem{corollary}[theorem]{Corollary}
 \newtheorem{lemma}[theorem]{Lemma}
 
 \theoremstyle{remark}

 \numberwithin {equation}{section}

\begin{document}
\title[Extension theorems]{Extension theorems for the Fourier transform
associated with non-degenerate
quadratic surfaces in vector spaces over finite fields}
\author{Alex Iosevich and Doowon Koh}

\address{Mathematics Department\\
202 Mathematical Sciences Bldg\\
University of Missouri\\
Columbia, MO 65211 USA}
\email{iosevich@math.missouri.edu}

\address{Mathematics Department\\
202 Mathematical Sciences Bldg\\
University of Missouri\\
Columbia, MO 65211 USA}

\email{koh@math.missouri.edu}



\begin{abstract}
We study the restriction of the Fourier transform to quadratic surfaces
in vector spaces over finite fields. In two dimensions, we obtain the
sharp result  by considering the sums of arbitrary two elements in the subset of quadratic
surfaces on two dimensional vector spaces over finite fields. For higher  dimensions,
we estimate the decay of the Fourier transform of the characteristic functions on quadratic surfaces 
so that we obtain the Tomas-Stein exponent. Using incidence theorems, 
we also study the extension theorems in the  restricted settings  
to sizes of sets in quadratic surfaces. Estimates for Gauss and Kloosterman
sums  and their variants play an important role.
\end{abstract}

\maketitle

\input epsf


\section{Introduction}
Let $S$ be a subset of ${\Bbb R}^d$ and $d\sigma$ a positive measure
supported on $S$. Then one may ask that for which values of $p$ and $r$ does the estimate
\begin{equation}\label{Boundedness} \|\widehat{fd\sigma}\|_{L^r({\Bbb R}^d)}
\le C_{p,r} \|f\|_{L^p(S,d\sigma)}\quad \text{for all} \quad 
 f \in L^p(S,d\sigma)\end{equation}
hold? This problem is known as the extension theorems.
See, for example, \cite{Fe70},\cite{Zy74},\cite{Ste93},\cite{Gr03},\cite{Ta03}, and the
references contained therein on recent progress related to this problem and its analogs.
In the case when $p=2$ in \eqref{Boundedness}, Strichartz (\cite{Str77}) gave a
complete solution when $S$  is a quadratic surface given by $S=\{x\in {\mathbb R}^d \colon Q(x)=j\}$,  
where $ Q(x)$ is a polynomial  of degree of two with real coefficients and $j$ is a real constant. 
In  this paper, we study the analogous  extension operators given by quadratic forms in the finite field setting,  
building upon earlier work of Mockenhaupt and Tao (\cite{MT04}) for the paraboloid in vector spaces  over finite fields.  
We begin with some notation and definitions to describe our main results.
Let ${\Bbb F}_q$ be a finite field of characteristic $char({\Bbb F}_q) > 2$ with $q$ elements, 
and let  ${\Bbb F}_q^d $ be a $d$-dimensional vector space over ${\Bbb F}_q$. Given  a function  
$f: {\Bbb F}_q^d \rightarrow {\Bbb C}, d\ge 1$,   
define the Fourier transform of $f$ by the  formula
\[ \widehat{f}(m)=q^{-d} \sum_{x\in {\Bbb F}_q^d}\chi(-x\cdot m) f(x)\]
where $\chi$ is a non-trivial additive character on ${\Bbb F}_q.$
When $\Bbb F_q = \Bbb Z/q\Bbb Z$ for some prime $q$ , we could take
$\chi(t)= e^{2\pi it/q}$, and the calculations in the paper are independent of the exact choice of the character.
Recall  that the Fourier inversion theorem is given by
\[f(x)=\sum_{m\in {\Bbb F}_q^d}\chi(x\cdot m)\widehat{f}(m).\] 
Also recall that the Plancherel theorem says in this context that
\[ \sum_{m\in {\Bbb F}_q^d} |\widehat{f}(m)|^2 =q^{-d}\sum_{x\in {\Bbb F}_q^d}|f(x)|^2.\]

Let $S \subset {\Bbb F}_q^d$ be an algebraic variety in ${\Bbb F}_q^d.$ We
denote by $d\sigma$ normalized surface measure on $S$ defined by the  relation
\[ \widehat{fd\sigma}(m) = \frac{1}{\#S}\sum_{x\in S} \chi(-x\cdot m)f(x),\] 
where $\# S$ denotes the number of elements in $S$. In other  words,
\[ q^{-d} \cdot \sigma(x)={(\# S)}^{-1} \cdot S(x).\]
Here, and throughout the paper, E(x) denotes the characteristic function , 
$\chi_E,$ of  the subset  $E$ of ${\Bbb F}_q^d.$ We therefore denote by $Ed\sigma$ the measure $\chi_Ed\sigma.$

For $1\le p,r <\infty$, define
\[\|f\|_{L^p \left({\Bbb F}_q^d,dx \right)}^p=q^{-d}\sum_{x\in {\Bbb 
F}_q^d}|f(x)|^p,\]
\[\|\widehat{f}\|_{L^r \left({\Bbb F}_q^d,dm \right)}^r = \sum_{m\in
{\Bbb F}_q^d}|\widehat{f}(m)|^r\]
 and
\[ \|f\|_{L^p \left(S,d\sigma \right)}^p=\frac{1}{\#S}\sum_{x\in S}|f(x)|^p.\] 
Similarly, denote by  $\|f\|_{L^\infty}$ the maximum value of $f$.

Observe that the measure on the "space" variables, $dx$, is the normalized
 measure obtained by dividing the counting measure by $q^d$, whereas the
 measure on the "phase" variables, $dm$, is just the usual counting
 measure. These normalizations are chosen in such a way that the Plancherel
 inequality takes the familiar form
\[{||\widehat{f}||}_{L^2({\Bbb F}_q^d,dm)}={||f||}_{L^2({\Bbb F}_q^d,dx)}.\]

We now define the non-degenerate quadratic surfaces in ${\Bbb F}_q^d$ in
the usual way. Let $x=(x_1,x_2,\cdots,x_d) \in {\Bbb F}_q^d.$ Denote by
$Q(x)$ a homogeneous polynomial in ${\Bbb F}_q[x_1,\cdots,x_d]$ of degree
2. Since $char({\Bbb F}_q) > 2$ throughout this paper, we can express
$Q(x)$ in the form
\[Q(x_1,x_2,\cdots,x_d) = \sum_{i,j=1}^d a_{ij}x_i x_j \quad \mbox{with}
\quad \ a_{ij}=a_{ji}.\]

If the $d \times d$ matrix $\{ a_{ij} \}$ is invertible, we say that the
Polynomial  $Q(x)$ is a non-degenerate quadratic form over ${\Bbb F}_q$. 
For each  $j\in {\Bbb F}_q^* ={\Bbb F}_q\setminus\{0\}$, the multiplicative group of
${\Bbb F}_q$, consider a set $S_j$ in ${\Bbb F}_q^d$ given by
\begin{equation}\label{surface}S_j=\{x\in {\Bbb F}_q^d: Q(x_1,\cdots,x_d)
=j\}, \end{equation}
where $Q(x)$ is a non-degenerate quadratic form. We call such a set $S_j$
a non-degenerate quadratic surface in ${\Bbb F}_q^d$. For example, the
sphere
\[S^{d-1}=\{x \in {\Bbb F}_q^d: x_1^2+x_2^2+\cdots +x_d^2=1\}\] 
is a non-degenerate quadratic surface in ${\Bbb F}_q^d.$

\subsection{Extension theorems and main results of this paper} Let $1\leq p,r\leq
\infty$. We define $R^*(p\rightarrow r)$ to be the best constant such that
the extension estimate
\[\|\widehat{fd\sigma}\|_{L^r ({\Bbb F}_q^d,dm)}\leq R^*(p\rightarrow r)
\|f\|_{L^p (S_j,d\sigma)}\] 
holds  for all functions $f$ on $S_j.$ The main goal of this paper is to
determine the set of exponents  $p$ and $r$ such that
\[R^*(p\rightarrow r)\leq C_{p,r}<\infty,\] 
where $C_{p,r}$ is independent  of the size of  ${\Bbb F}_q$. We note that
\begin{equation}\label{fact} R^*(p_1\rightarrow r) \le R^*(p_2\rightarrow
r) \quad \mbox{for}\quad \ p_1\ge p_2,\end{equation}
and
\[R^*(p\rightarrow r_1)\le R^*(p\rightarrow r_2) \quad \mbox{for}\quad \
r_1\ge r_2, \] 
which will allow us to reduce the analysis below to certain  endpoint estimates.

Let $S$ be an algebraic variety in ${\Bbb F}_q^d$ with $\#S \approx q^k$
for  some $0<k<d$. Here, and throughout the paper, 
$X \lesssim Y$ means that there exists $C>0$, independent of
$q$ such that $X \leq CY$, and $X \approx Y$ means both $X \lesssim Y$ and $Y \lesssim X.$  
Mockenhaupt and Tao  (\cite{MT04}) proved that $R^*(p\to r)$ is uniformly bounded  
($O(1)$ with constants independent of the size of ${\Bbb F}_q$) only if
\begin{equation}\label{conjecture} r\ge \frac{2d}{k}  \quad\mbox{and}\quad
r\ge  \frac{dp}{k(p-1)}.\end{equation}
For the detailed proofs of these assertions, see (\cite{MT04}, pages 41-42). 

Mockenhaupt and Tao also showed that $R^*(2\rightarrow r)$ is uniformly
bounded whenever 
\begin{equation}\label{Tomas-Stein} r\ge \dfrac{2d+2}{d-1} \end{equation}
if
\[S=\{(x,x\cdot x): x\in F_q^{d-1}\},\] 
an analog of the Euclidean  paraboloid. Moreover, when $d=3$ and $-1$ is not a square 
in ${\Bbb F}_q$, they  improved the result in (\ref{Tomas-Stein}) by showing that 
for each  $\varepsilon >0$ there exists $ C_{\epsilon}>0$ such that  
\begin{equation} \label{mtincidence} R^*\left(\frac{8}{5} \rightarrow 4  \right) \lesssim 1 
\quad  \mbox{and}\quad R^*\left(2\rightarrow \frac{18}{5} \right)
\leq C_{\epsilon} q^{\epsilon}. \end{equation}

If we replaced the paraboloid by a general non-degenerate quadratic
surface, the extension problem becomes more complicated, in part because
the Fourier transform of quadratic surfaces cannot be computed by simply
considering the Gauss sums, as was pointed out by the authors in
\cite{MT04}. Using generalized Kloosterman sums, we estimate the decay of
the Fourier transform of non-degenerate  quadratic surfaces. As a result, 
we obtain Theorem \ref{Main3} below which  gives the same exponents as in (\ref{Tomas-Stein})
(see the FIGURE \ref{Picture2}).
\begin{theorem}\label{Main3} Let $S_j$ be a non-degenerate quadratic
surface in  ${\Bbb F}_q^d$ defined as in (\ref{surface}). If $d\ge2$ and $r\ge
\frac{2d+2}{d-1}$, then
\[R^*(2\rightarrow r) \lesssim 1.\] \end{theorem}
In the case when $d=2$, Mochenhaupt and Tao  (\cite{MT04}) showed that the
necessary conditions for the boundedness of $R^*(p\to r)$ in  (\ref{conjecture}) 
are also sufficient when $S$ is the parabola. Theorem  \ref{Main2} below implies that 
this also holds in the case when $S$ is a  non-degenerate quadratic  curve. 
To see this, observe from Corollary \ref{cor1} that $\#S \approx q$  for $d=2.$ 
Thus the necessary conditions in (\ref{conjecture}) take the  form
\begin{equation} \label{necessary2d} r\ge 4 \quad \mbox{and}\quad r\ge
\frac{2p}{p-1}.\end{equation}
Combining (\ref{fact}) with Theorem \ref{Main2} below, we see that
\begin{equation}\label{Trivial1} R^*(p\rightarrow 4)\lesssim 1 \quad
\mbox{for} \quad 2 \le p\le \infty.\end{equation}
By direct estimation, we have
 \begin{equation}\label{Trivial2}  R^*(p\to \infty) \lesssim 1 \quad \mbox{for}\quad 1\le  p\le \infty.\end{equation}
Interpolating (\ref{Trivial1}) and (\ref{Trivial2}), we see that the
necessary conditions given by (\ref{necessary2d}) are in fact sufficient
as we claim once we establish the following result.
\begin{theorem}\label{Main2} Let $d\ge2 $. Let $S_j$ be the non-degenerate
quadratic surface in  ${\Bbb F}_q^d$ defined as in (\ref{surface}).  Then we have
\[ R^*(2\rightarrow 4) \lesssim 1.\] \end{theorem}
Observe that Theorem \ref{Main2} is stronger in two dimensions. Theorem
\ref{Main3} and Theorem \ref{Main2} are the same in three dimensions, and
Theorem \ref{Main3} is stronger in dimension four and higher.
Using incidence theory, we are able to improve the exponents above in a
restricted setting. See Theorem \ref{Incidence1} and FIGURE \ref{Picture2} 
for extension theorems restricted to big sets, and also see Theorem \ref{Incidence3}, 
FIGURE \ref{Picture3}, \ref{Picture4}, and \ref{Picture5} for extension theorems restricted to small sets. 
These results are analogous to those obtained by
Mockenhaupt and Tao as described in (\ref{mtincidence}) above. While the
aforementioned authors use combinatorial methods to prove their incidence
theorems, we use Fourier analytic methods which eventually reduce proofs
to the estimates for Kloosterman and related sums.
\begin{theorem}\label{Incidence1}Let $S_j$ be a non-degenerate quadratic
surface in ${\Bbb F}_q^d$ and $E$ be a subset of $S_j.$
Then we have the following estimate
\begin{equation}\label{Bigsize}\|\widehat{E d\sigma}\|_{L^4({\Bbb
F}_q^d,dm)} \lesssim  \|E\|_{L^{\frac{4}{3}} (S_j, d\sigma)} \quad
\mbox{for} \quad  q^{\frac{d+1}{2}} \lesssim \#E \lesssim
q^{d-1}.\end{equation}
\end{theorem}

\begin{figure}[h]
\centering\leavevmode\epsfysize=4.5cm \epsfbox{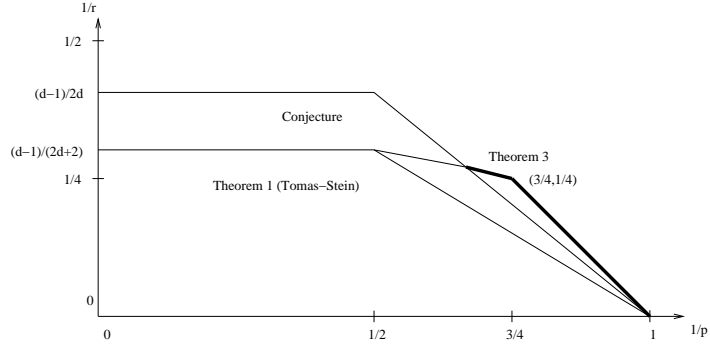}
\caption{\label{Picture2} Tomas-Stein exponent and extension estimates in a restricted setting to big sets
$( q^{\frac{d+1}{2}}\lesssim \#E \lesssim q^{d-1})$}
\end{figure}

\begin{theorem}\label{Incidence3}
Let $S_j$ be a non-degenerate quadratic surface in ${\Bbb F}_q^d$ and $E$
be a subset of $S_j.$  Then for every $p_0 \geq 2,$ we have the following  estimates
\begin{equation} \|\widehat{E d\sigma}\|_{L^r({\Bbb F}_q^d,dm)} \lesssim
\|E\|_{L^{p}(S_j, d\sigma)} \quad \mbox{for} \quad
1 \lesssim \#E \lesssim q^{\frac{d-1}{2}} \end{equation}
where the exponents $p$ and $r$ are given by
\[ p\ge \frac{ (6d-2)p_0 -8d+8}{(3d-5)p_0-4d+12} \qquad \mbox{and} \qquad
r \ge \frac{ (6d-2)p_0 -8d+8}{(3d-3)p_0-4d+4},\]
and
\begin{equation} \|\widehat{E d\sigma}\|_{L^r({\Bbb F}_q^d,dm)} \lesssim
\|E\|_{L^{p}(S_j, d\sigma)} \quad \mbox{for} \quad
1 \lesssim \#E \lesssim q^{\frac{d+1}{2}} \end{equation}
where the exponents $p$ and $r$ are given by
\[ p\ge \frac{ (6d-10)p_0 -8d+24}{(3d-9)p_0-4d+20} \qquad \mbox{and}
\qquad r \ge \frac{ (6d-10)p_0 -8d+24}{(3d-7)p_0-4d+12}.\] \end{theorem}

\begin{figure}[h]
\centering\leavevmode\epsfysize=4.5cm \epsfbox{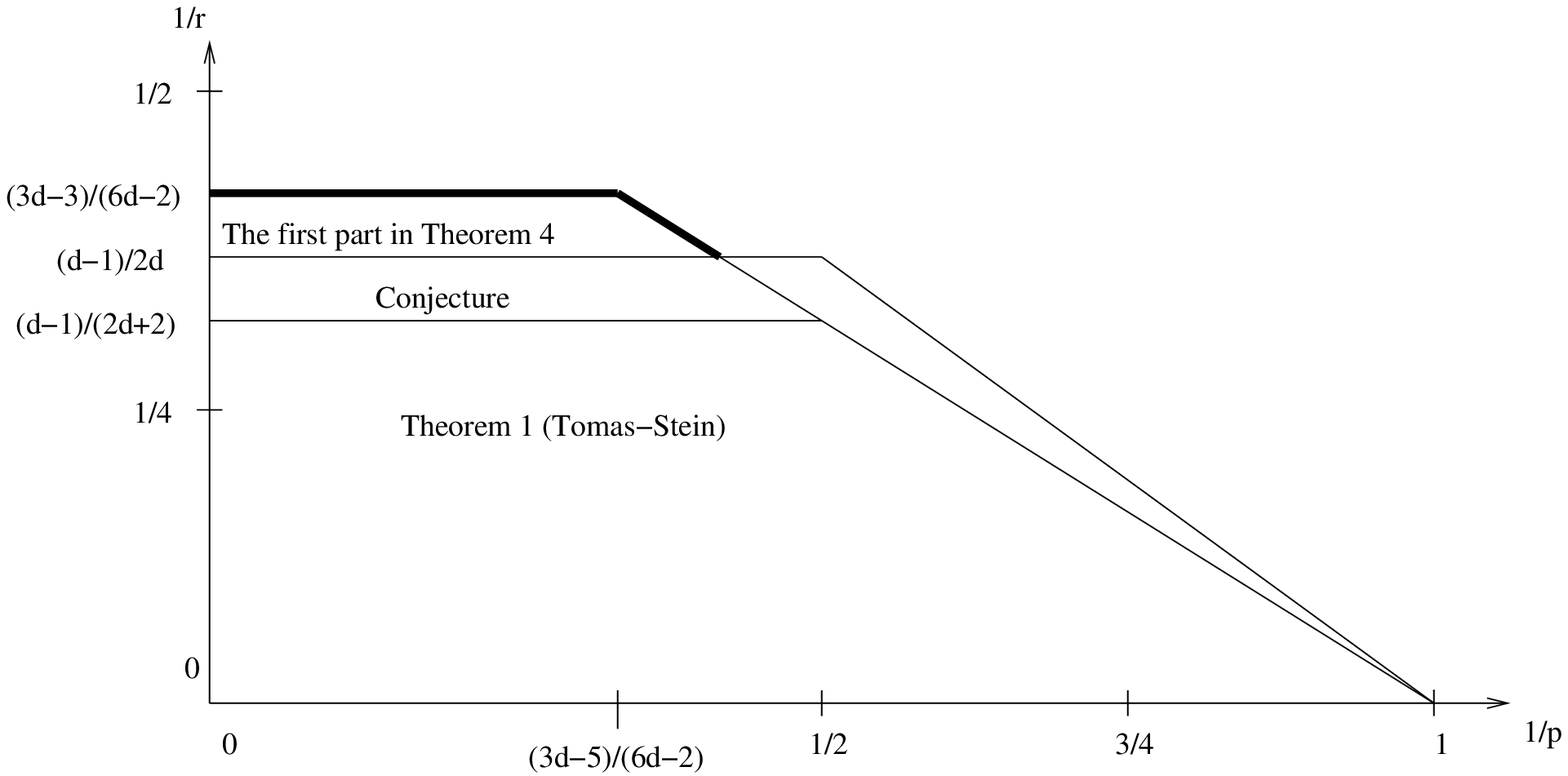}
\caption{\label{Picture3} Extension estimates in a restricted setting 
\newline
$( 1\lesssim \#E \lesssim q^{\frac{d-1}{2}})$}
\end{figure}

\begin{figure}[h]
\centering\leavevmode\epsfysize=4.5cm \epsfbox{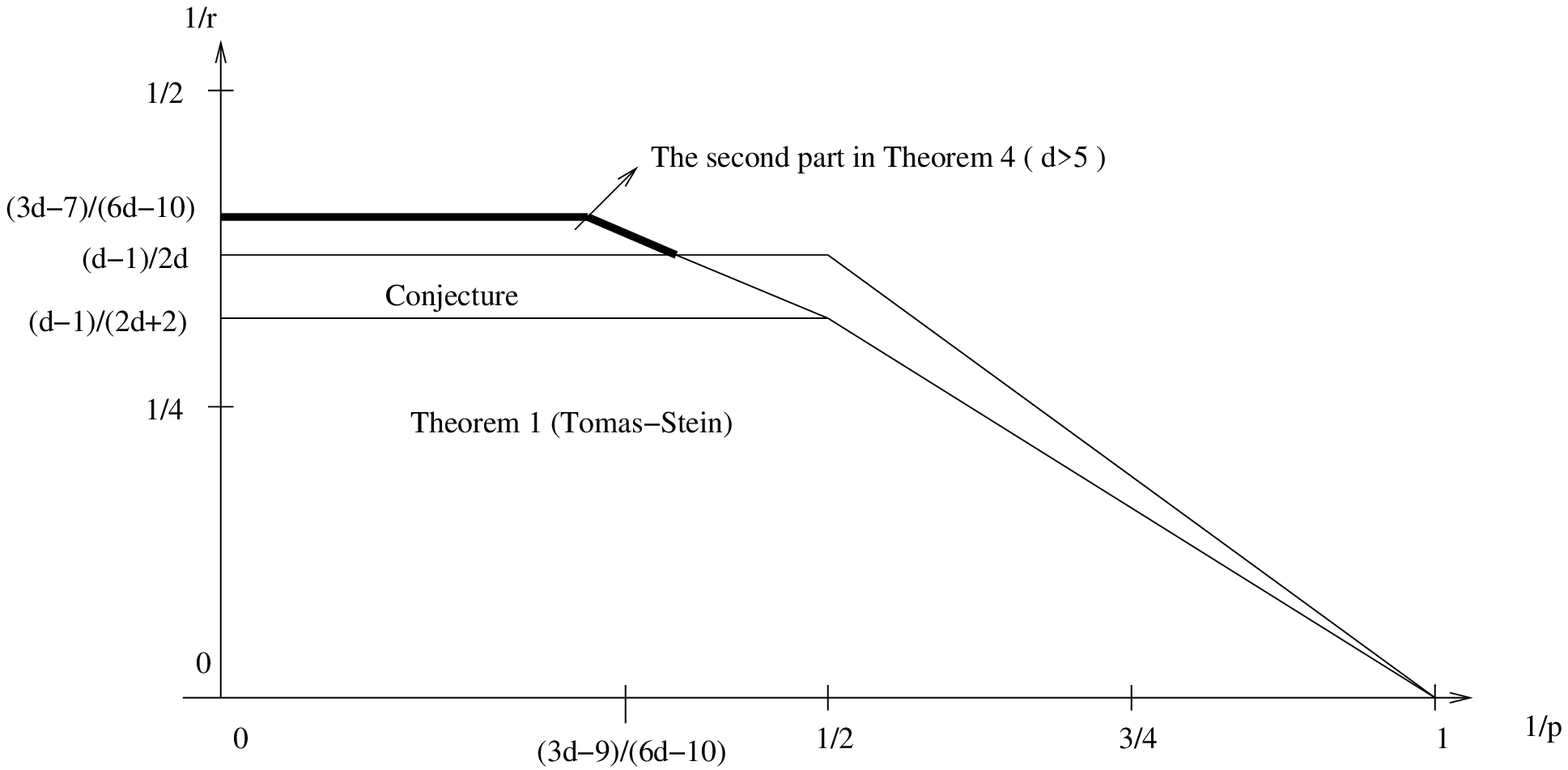}
\caption{\label{Picture4} Extension estimates in a restricted setting 
\newline
$( 1\lesssim \#E \lesssim q^{\frac{d+1}{2}} , d >5)$}
\end{figure}

\begin{figure}[h]
\centering\leavevmode\epsfysize=4.5cm \epsfbox{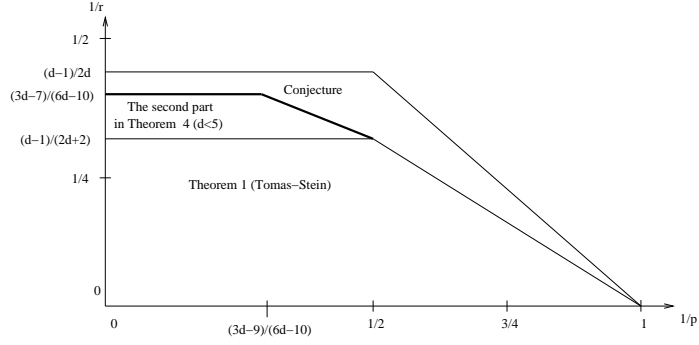}
\caption{\label{Picture5} Extension estimates in a restricted setting 
\newline
$( 1\lesssim \#E \lesssim q^{\frac{d+1}{2}}, d<5)$}
\end{figure}
\subsection{Outline of this paper} In section 2, we shall introduce few
theorems related to bounds on exponential sums. As an application,
we get the decay of the Fourier transform of the characteristic functions on
the non-degenerate quadratic surfaces in vector spaces over finite fields
(see Lemma \ref{keylemma1} below). In section 3, we shall prove Theorem \ref{Main3} 
which can be obtained from the results of Lemma \ref{keylemma1}. In section 4,
the proof of Theorem \ref{Main2} will be given. In the final section, we prove 
Theorem \ref{Incidence1} and  Theorem \ref{Incidence3}. 
\section{Classical bounds on exponential sums and consequences}
In this section, we shall estimate the decay of Fourier transform of 
the characteristic functions on non-degenerate quadratic surfaces in $\Bbb{F}_q^d$
using the classical bounds on exponential sums. 
To do this, we first introduce the well known theorems for exponential sums.
The following theorem is a well known estimate for Gauss sums.
\begin{theorem}\label{Gausssum}Let $\chi$ be a non-trivial additive
character of ${\Bbb F}_q$,
and $\psi$ a multiplicative character of ${\Bbb F}_q^*.$ It follows that
\[G_a(\chi,\psi)=\sum_{t\in {\Bbb  F}_q^*}\chi(at)\psi(t)=O(q^{\frac{1}{2}}), 
\ a\in {\Bbb F}_q^*.\] \end{theorem}
\begin{proof}If $\psi = 1$, the result is obvious, so we may assume that
$\psi$ is a non-trivial multiplicative character of ${\Bbb F}_q^*$. We
Have  \begin{align*}|G_a(\chi,\psi)|^2 &=\sum_{t\in {\Bbb F}_q^*}\sum_{s\in
{\Bbb F}_q^*}\chi(at-as)\psi(ts^{-1})\\
&=\sum_{t\in {\Bbb F}_q^*}\psi(t)\sum_{s\in {\Bbb F}_q^*}\chi(ast-as)\\
&=\sum_{t\in {\Bbb F}_q^*}\psi(t)\Big(-1+\sum_{s\in {\Bbb
F}_q}\chi(st-s)\Big)\\
&=-\sum_{t\in {\Bbb F}_q^*}\psi(t) +\sum_{t\in {\Bbb F}_q^*}\sum_{s\in
{\Bbb F}_q}\chi((t-1)s)\\
&=0+q =q.\end{align*}
Thus
 \[|G_a(\chi,\psi)|= q^{\frac{1}{2}}\] and the proof is complete.  \end{proof}
The following theorem gives us the relation between more general
exponential sums and Gauss sums in Theorem \ref{Gausssum}. For a nice
proof, see \cite{LN97}.
\begin{theorem}\label{sumfomula} Let $\chi$ be a non-trivial additive
character of ${\Bbb F}_q, n\in {\mathbb N},$ and $\psi$ a multiplicative
character of ${\Bbb F}_q^*$ of order $h = gcd (n, q-1).$ Then
\[\sum_{s\in {\Bbb F}_q}
\chi(ts^n)=\sum_{k=1}^{h-1}\psi^{-k}(t)G(\psi^k,\chi)\] 
for any $t \in  {\Bbb F}_q^*,$ where $G(\psi^k,\chi)=\sum\limits_{s\in {\Bbb F}_q^*}
\psi^k(s)\chi(s).$  \end{theorem}
The following theorem is well known as the estimation of the Sali\'e' sum,
often referred to as the twisted Kloosterman sum. See \cite{Sa32}.
\begin{theorem}\label{salie} Let $\psi$ be a multiplicative character of
order two of ${\Bbb F}_q^*$, $q$ odd, and $a,b\in {\Bbb F}_q$. Then for
any additive character $\chi$ of ${\Bbb F}_q,$
\[ \left|\sum_{t\in {\Bbb F}_q^*}\psi(t)\chi(at+bt^{-1}) \right| \lesssim
q^{\frac{1}{2}}. \] \end{theorem}
The following is a classical estimate for Kloosterman sums due to Wey
(\cite{We48}). See also \cite{LN97}.
\begin{theorem}\label{Kloosterman} If $\chi$ is a non-trivial additive
character of  ${\Bbb F}_q$ and $a,b \in {\Bbb F}_q$ are not both zero, then
we have
\[ \left| \sum_{t\in {\Bbb F}_q^*}\chi(at+bt^{-1}) \right| \lesssim
q^{\frac{1}{2}}.\] \end{theorem}
With the same notation as above, we have the following estimate on the
Fourier transform of the characteristic function of a non-degenerate
quadratic surface.
\begin{lemma}\label{keylemma1} Let ${\Bbb F}_q, q$ odd, be a finite field.
Then
\[|\widehat{S_j}(m)|=\Big| q^{-d}\sum_{x\in S_j}\chi(-x\cdot m)\Big|
\lesssim q^{-\frac{d+1}{2}}\] 
if $m\ne (0,\cdots,0)$, and
\[\widehat{S_j}(0,\cdots,0) \approx q^{-1}.\] \end{lemma}
From Lemma \ref{keylemma1}, we obtain the following corollary.
\begin{corollary}\label{cor1} \[\#S_j \approx q^{d-1}.\]\end{corollary}
\begin{proof} Using the second part of Lemma \ref{keylemma1}, we have
\[\widehat{S_j}(0,\cdots,0)=q^{-d}\sum_{x\in {\Bbb F}_q^d}S_j(x)\approx
q^{-1},\] and the result follows. \end{proof}

 \subsection{Proof of Lemma \ref{keylemma1}} We first observe that
 \begin{align*}\widehat{S_j}(m)&= q^{-d}\sum_{x\in S_j}\chi(-x\cdot m)\\
 &=q^{-d}\sum_{x\in {\Bbb F}_q^d}\chi(-x\cdot m) q^{-1} \sum_{t\in {\Bbb
 F}_q} \chi(t(Q(x)-j))\\
 &=q^{-1}\delta_0(m) + q^{-d-1} \sum_{t\in {\Bbb F}_q^*}\chi(-jt)\sum_{x
 \in {\Bbb F}_q^d}\chi(tQ(x)-x\cdot m)\end{align*} where $\delta_0(m) = 1$
 if $m=(0,\cdots,0)$ and $\delta_0(m)=0$ otherwise. To complete the
 proof of Lemma \ref{keylemma1}, it suffices to show that for $j\ne
 0, m \in {\Bbb F}_q^d$,
 \begin{equation}\label{k1}D(j,m)\lesssim q^{\frac{d+1}{2}}\end{equation}
 where
 \begin{equation}\label{Substuton}D(j,m) =\sum_{t\in {\Bbb
 F}_q^*}\chi(-jt)\sum_{x \in {\Bbb F}_q^d}
 \chi(tQ(x)-x\cdot m) .\end{equation} Let
 \[ W_t(m)= \sum_{x \in {\Bbb F}_q^d}\chi(tQ(x)-x\cdot m)\]
 for $m\in {\Bbb F}_q^d , t \in {\Bbb F}_q^* $. We shall need the following
 theorem (see \cite{LN97}).
 \begin{theorem}\label{change}
 Every quadratic form $Q(x)=\sum\limits_{i,k=1}^d a_{ik}x_ix_k $
 over ${\Bbb F}_q, q $ odd, can be transformed into a diagonal form
 $a_1x_1^2+\cdots+a_d x_d^2$ over ${\Bbb F}_q$ by means of a nonsingular
 linear substitution of indeterminates. Moreover if $Q(x)$ is a
 non-degenerate quadratic form, then $a_i \ne 0$ for all
 $i=1,2,\cdots,d.$\end{theorem} Using Theorem \ref{change}, we may write
 that for some $m^{\prime}=(m_1^{\prime},\cdots,m_d^{\prime}) \in
 {\Bbb F}_q^d $, and $a_i\in {\Bbb F}_q^*$  for all $i=1,2,\cdots,d $,
 \[W_t(m)=\sum_{x \in {\Bbb F}_q^d}\chi(t\|x\|_a + x \cdot m^{\prime}),\]
 where $m^{\prime} \in {\Bbb F}_q^d$ is  determined by $ m\in {\Bbb F}_q^d$
 and
 $\|x\|_a$ is given by
 \[\|x\|_a =a_1x_1^2+\cdots+ a_d x_d^2.\]
 Since $\chi$ is an additive character of ${\Bbb F}_q$, we have
 \begin{align*}W_t(m)&=\prod_{k=1}^d \sum_{x_k\in {\Bbb F}_q}\chi(ta_k
 x_k^2+ m_k^{\prime}x_k)\\
 &=\prod_{k=1}^d\sum_{x_k\in {\Bbb
 F}_q}\chi(ta_k(x_k+(2ta_k)^{-1}m_k^{\prime})^2-(4ta_k)^{-1}{m_k^{\prime}}^2)\\
 &=\prod_{k=1}^d\chi(-(4ta_k)^{-1}{m_k^{\prime}}^2)\sum_{x_k\in
 {\Bbb F}_q}\chi(ta_k x_k^2).\end{align*} Using Theorem \ref{sumfomula},
 we see that
 \[\sum_{x_k\in {\Bbb F}_q}\chi(ta_k x_k^2) = \psi^{-1}(ta_k)G(\chi,\psi)\]
 where $\psi$ is a multiplicative character of ${\Bbb F}_q^*$ of order two
 and $G(\chi,\psi)= \sum\limits_{s\in {\Bbb F}_q^*}\chi(s)\psi(s).$ Thus we
 obtain that
 \begin{align}\label{k2}W_t(m)&=\psi^{-d}(t)\psi^{-1}(a_1\cdots a_d)
 (G(\chi,\psi))^d\prod_{k=1}^d\chi(-(4ta_k)^{-1}{m_k^{\prime}}^2)\nonumber\\
 &=\psi^{-d}(t)\psi^{-1}(a_1\cdots a_d)(G(\chi,\psi))^d
 \chi(t^{-1}\sum_{k=1}^d-(4a_k)^{-1}{m_k^{\prime}}^2).
 \end{align}
 Combining  above fact in (\ref{k2}) with (\ref{Substuton}), we obtain that
 \begin{equation}\label{k3}D(j,m)=\psi^{-1}(a_1\cdots a_d)(G(\chi,\psi))^d
 \sum_{t\in {\Bbb F}_q^* }\chi(-jt+t^{-1}M)\psi^{-d}(t),\end{equation}
 where $M$ is given by
 \[ M= \sum_{k=1}^d-(4a_k)^{-1}{m_k^{\prime}}^2.\]

 Since $\psi$ is a multiplicative character of order two, we see that
 $\psi^{-d}=1$ for d even, and $\psi^{-d}=\psi$ for d odd. Therefore, in
 order to get the inequality in (\ref{k1}) , we can apply Theorem
 \ref{Gausssum} and \ref{salie} to (\ref{k3}) for $d$  odd. On the other
 hand, if $d$ is even, we can apply Theorem \ref{Gausssum} and
 \ref{Kloosterman} to (\ref{k3}) because $j\ne 0.$ This completes the proof
 of Lemma \ref{keylemma1}.

 \section{Proof of the Tomas-Stein exponent (Theorem \ref{Main3})}
 Theorem \ref{Main3} is a result from Lemma \ref{keylemma1} in this paper 
and Lemma 6.1  in \cite{MT04}. We first introduce Lemma 6.1 in \cite{MT04}. Let $S$
 be an algebraic variety in ${\Bbb F}_q^d $ with a
 normalized surface measure $d\sigma$. We introduce the
 Bochner-Riesz kernel
 \[K(m) :=\widehat{d\sigma}(m)-\delta_0(m)\]
 where $\delta_0(m)=1$ if $m=(0,\cdots,0)$ and $\delta_0(m)=0$
 otherwise. We need the following theorem. For a nice proof, see
 Lemma 6.1 in \cite{MT04}.
 \begin{theorem}\label{Tao} Let $p,r \ge 2,$ and $ {\Bbb F}_q^d $ be a
 $d$-dimensional vector space over ${\Bbb F}_q.$
 Suppose that
 \[\|K\|_{L^{\infty}({\Bbb F}_q^d, dm)} = \|
 \widehat{d\sigma}-\delta_0\|_{L^{\infty}({\Bbb F}_q^d, dm)}
 \lesssim q^{-\frac{\tilde{d}}{2}}\] 
for some $0< \tilde{d} < d.$
 Then for any $0<\theta < 1$, we have
 \[R^*\left(p\rightarrow \frac{r}{\theta}\right) \lesssim 1
 +R^*\left(p\rightarrow r \right)^\theta
 q^{-\frac{\tilde{d}(1-\theta)}{4}}.\]\end{theorem}
 We are now ready to prove Theorem \ref{Main3}. Recall that
 we are working with a non-degenerate quadratic surface $S_j$ in ${\Bbb
 F}_q^d.$ We now check that \begin{equation}\label{k9}
 \|K\|_{L^{\infty}({\Bbb F}_q^d,dm)} \lesssim
 q^{-\frac{(d-1)}{2}}.\end{equation}
 In fact, if $m=(m_1, m_2, \cdots,m_d) \ne (0,\cdots,0)$ then we
 have
 \begin{align*} K(m) = \widehat{d\sigma}(m)
 &=(\#S_j)^{-1}\sum_{x\in S_j}\chi(-x\cdot m)\\
 &=(\#S_j)^{-1}\sum_{x \in {\Bbb F}_q^d}\chi(-x\cdot m)S_j(x)\\
 &=(\#S_j)^{-1}q^d \widehat{S_j}(m).
 \end{align*}
 From Corollary \ref{cor1} and Lemma \ref{keylemma1}  , we have
 \[ (\#S_j)\approx q^{d-1} \quad \mbox{and}\quad
 |\widehat{S_j}(m)|\lesssim q^{-\frac{d+1}{2}} \quad \mbox{for}
 \,\, m \ne (0,\cdots,0).\] We therefore obtain that
 \[ |K(m)| \lesssim q^{-\frac{d-1}{2}} \quad \mbox{for}\,\, m\ne
 (0,\cdots,0).\]
 On the other hand, we have
 \[K(0,\cdots,0) = \widehat{d\sigma}(0,\cdots,0) -1 =0.\]
 Thus the inequality in (\ref{k9}) holds. We now claim that
 \begin{equation}\label{k10} R^*(2\rightarrow 2) \approx q^{\frac{1}{2}}.
 \end{equation}
 To justify above claim, we shall show that
 \begin{equation}\label{Tag14}\|\widehat{fd\sigma} \|_{L^2({\Bbb
 F}_q^d,dm)}\approx q^{\frac{1}{2}} \|f\|_{L^2(S_j,d\sigma)}\end{equation}
 for all functions $f$ on $S_j.$ We first note that
 \begin{align*}|\widehat{fd\sigma}(m)|^2 &=(\#S_j)^{-2}\sum_{x\in
 S_j}\chi(-x\cdot m)  f(x)\sum_{y\in S_j}\chi(y\cdot m) \overline{f(y)}\\
 &=(\#S_j)^{-2}\sum_{x, y \in S_j} \chi((y-x)\cdot
 m)f(x)\overline{f(y)}.\end{align*} We have
 \begin{align*} \|\widehat{fd\sigma} \|_{L^2({\Bbb F}_q^d,dm)}
 &= \Big(\sum_{m\in {\Bbb F}_q^d} |\widehat{fd\sigma}
 (m)|^2\Big)^{\frac{1}{2}}\\
 &=(\#S_j)^{-1}\Big( \sum_{x, y \in S_j} \sum_{m\in {\Bbb
 F}_q^d}\chi((y-x)\cdot m)f(x)\overline{f(y)}\Big)^{\frac{1}{2}}\\
 &=(\#S_j)^{-1}q^{\frac{d}{2}}\Big(\sum_{x \in
 S_j}|f(x)|^2\Big)^{\frac{1}{2}}\\
 &=(\#S_j)^{-1}q^{\frac{d}{2}}(\#S_j)^{\frac{1}{2}}\|f\|_{L^2(S_j,d\sigma)}
 \approx q^{\frac{1}{2}}\|f\|_{L^2(S_j,d\sigma)}\end{align*} In the
 last equality, we used the fact that $ \#S_j \approx q^{d-1}$. Thus
 our claim in (\ref{k10}) is proved. Using Theorem \ref{Tao} with
 (\ref{k9}) and (\ref{k10}), we obtain that for any $0<\theta<1$,
 \begin{align*} R^*(2\rightarrow \frac{2}{\theta})&
 \lesssim 1+ R^*(2\rightarrow 2)^\theta q^{-\frac{(d-1)(1-\theta)}{4}}\\
  &\lesssim
 1+q^{\frac{\theta}{2}}q^{-\frac{(d-1)(1-\theta)}{4}}.\end{align*}
 Taking $0<\theta\le\frac{d-1}{d+1},$ we have
 \[R^*(2\rightarrow\frac{2}{\theta})\lesssim 1.\]
 Thus Theorem \ref{Main3} is proved with $r=\frac{2}{\theta}.$

 \section{Proof of the $L^2 \rightarrow L^4$ estimate (Theorem
 \ref{Main2})}
 To prove Theorem \ref{Main2}, we make the following reduction.
 \begin{lemma}\label{sum} Let $S_j$ be a non-degenerate quadratic surface
 in
 ${\Bbb F}_q^d$ defined as in (\ref{surface}). Suppose that for any
 $x\in({\Bbb F}_q^d)^* ={\Bbb F}_q^d\setminus (0,\cdots , 0)$, we have
 \[\sum_{\{(\alpha,\beta)\in S_j\times S_j: \alpha+\beta =x\}} 1 \lesssim
 q^{d-2}.\]
 Then for $d\ge 2$,
 \[R^*(2\rightarrow 4) \lesssim 1.\]
 \end{lemma}
 \begin{proof} We have to show that
 \[\|\widehat{fd\sigma}\|_{L^4({\Bbb F}_q^d,dm)} \lesssim
 \|f\|_{L^2(S_j,d\sigma)}\]
 for all functions $f$ on $S_j.$ Using Plancherel, we have
 \begin{align*} \|\widehat{fd\sigma}\|_{L^4({\Bbb F}_q^d,dm)}
 &=\|\widehat{fd\sigma}\widehat{fd\sigma}\|^{\frac{1}{2}} _{L^2({\Bbb
 F}_q^d,dm)} \\
 &=\|fd\sigma \ast fd\sigma\|^{\frac{1}{2}} _{L^2({\Bbb F}_q^d,dx)}
 \end{align*} and so it suffices to show that
 \[\|fd\sigma \ast fd\sigma\|^2_{L^2({\Bbb F}_q^d,dx)} \lesssim
 \|f\|^4_{L^2(S_j,d\sigma)}.\]
 It follows that
 \[\|fd\sigma \ast fd\sigma\|^2_{L^2({\Bbb F}_q^d,dx)}\]
 \[=q^{-d}|fd\sigma \ast fd\sigma (0,\cdots,0)|^2 + \|fd\sigma \ast
 fd\sigma\|^2_{L^2(({\Bbb F}_q^d)^*,dx)}\]
 Thus it will suffice to show that
 \begin{equation}\label{tag1} q^{-d}|fd\sigma \ast fd\sigma
 (0,\cdots,0)|^2\lesssim \|f\|^4_{L^2(S_j,d\sigma)}\end{equation}
 and
 \begin{equation}\label{tag2}\|fd\sigma \ast fd\sigma\|^2_{L^2(({\Bbb
 F}_q^d)^*,dx)}\lesssim \|f\|^4_{L^2(S_j,d\sigma)}\end{equation}
 We first show that the inequality in (\ref{tag1}) holds. We have
 \begin{align*}|fd\sigma \ast fd\sigma (0,\cdots,0)|&\le \sum_{m \in {\Bbb
 F}_q^d } |\widehat{fd\sigma}(m)|^2 \\
 &=(\#S_j)^{-2}q^d\sum_{x\in S_j}|f(x)|^2\\
 &=(\#S_j)^{-1}q^d \|f\|^2_{L^2(S_j,d\sigma)} \approx q
 \|f\|^2_{L^2(S_j,d\sigma)}.
 \end{align*}
 Thus the inequality in (\ref{tag1}) holds because $d\ge2.$
 It remains to show that the inequality in (\ref{tag2}) holds.
 Without loss of generality, we may assume that $f$ is positive. Using the
 Cauchy Schwartz  inequality, we see that
 \begin{align} \label{tag3}
 &fd\sigma \ast fd\sigma(x) \\
 &= (\#S_j)^{-2}q^d \sum_{\{(\alpha,\beta)\in S_j\times S_j :\alpha+\beta
 =x\}} f(\alpha)f(\beta)\nonumber\\
 &\le(\#S_j)^{-2}q^d \Big(\sum_{\{(\alpha,\beta)\in S_j\times S_j
 :\alpha+\beta =x\}} 1\Big)^{\frac{1}{2}}
 \Big(\sum_{\{(\alpha,\beta)\in S_j\times S_j :\alpha+\beta
 =x\}}f^2(\alpha)f^2(\beta)\Big)^{\frac{1}{2}}\nonumber\\
 &=(d\sigma \ast d\sigma)^{\frac{1}{2}}(x) (f^2d\sigma \ast
 f^2d\sigma)^{\frac{1}{2}}(x).\nonumber \end{align} From our hypothesis and
 the fact that $\#S_j \approx q^{d-1}$, we obtain that for $x \ne
 (0,\cdots,0),$
 \begin{equation}\label{tag4}d\sigma \ast d\sigma(x) \approx q^{-d+2}
  \sum_{\{(\alpha,\beta)\in S_j\times S_j: \alpha+\beta =x\}} 1 \lesssim
 1.\end{equation}
 From Fubini's theorem, we also have
 \begin{equation}\label{tag5} \|f^2d\sigma \ast f^2d\sigma\|_{L^1({\Bbb
 F}_q^d,dx)}=\|f\|^4_{L^2(S_j,d\sigma)}.
 \end{equation}
 Using H\"older inequality and estimates (\ref{tag3}), (\ref{tag4}) , and
 (\ref{tag5}), we obtain that
 \begin{align*}\|fd\sigma \ast fd\sigma\|^2_{L^2(({\Bbb F}_q^d)^*,dx)}
  &=\|(fd\sigma \ast fd\sigma)^2\|_{L^1(({\Bbb F}_q^d)^*,dx)}\\
 &\le\|(d\sigma \ast d\sigma)\cdot (f^2d\sigma \ast
 f^2d\sigma)\|_{L^1(({\Bbb F}_q^d)^*,dx)}\\
 &\le \|d\sigma \ast d\sigma\|_{L^\infty(({\Bbb F}_q^d)^*,dx)}\|f^2d\sigma
 \ast f^2d\sigma\|_{L^1(({\Bbb F}_q^d)^*,dx)}\\
 &\lesssim \|f\|^4_{L^2(S_j,d\sigma)}.
 \end{align*}
 Thus the inequality in (\ref{tag2}) holds and so the proof of
 Lemma \ref{sum} is complete.\end{proof} We now prove Theorem
 \ref{Main2}. By Lemma \ref{sum}, it is enough to show that for any
 $x \in ({\Bbb F}_q^d)^*, d\ge 2,$
 \begin{equation}\label{tag6} \sum_{\{(\alpha,\beta)\in S_j\times S_j:
 \alpha+\beta =x\}} 1 \lesssim q^{d-2}\end{equation}
 where  $S_j$ is the non-degenerate quadratic surface in ${\Bbb
 F}_q^d$. Using Theorem \ref{change}, we may assume that the
 non-degenerate quadratic surface in ${\Bbb F}_q^d$ is given by
 \[ S_j=\{y\in {\Bbb F}_q^d: a_1y_1^2 +\cdots+a_d y_d^2=j\ne 0\}\]
 for all $a_k\ne 0,\,\,\,k=1,2,\cdots,d.$ Therefore the left hand side of
 the equation in (\ref{tag6}) can be estimated by
 the number of common solutions $\alpha=(\alpha_1,\cdots,\alpha_d)$ in
 ${\Bbb F}_q^d$ of the equations
 \begin{align}\label{tag7} a_1\alpha_1^2+\cdots+a_d\alpha_d^2&=j\nonumber\\
 2a_1x_1\alpha_1+\cdots+ 2a_d x_d\alpha_d&=\sum_{k=1}^d a_k x_k^2
 \end{align}
 for $x=(x_1,\cdots,x_d)\ne(0,\cdots,0)$ and $a_k\ne 0$ for all
 $k=1,2,\cdots,d.$ Note that $2a_k x_k\ne 0$ for some
 $k=1,2,\cdots,d$ because $ x \ne (0,\cdots,0)$ and $a_k\ne 0.$
 Thus a routine algebraic computation shows that the number of
 common solutions of equations in (\ref{tag7}) is less than equal
 to $2q^{d-2}.$ This means that the inequality in (\ref{tag6})
 holds and so we complete the proof of Theorem \ref{Main2}.

 \section{Incidence theorems and the proof of Theorem \ref{Incidence1} and
 Theorem \ref{Incidence3}}
 The purpose of this section is to develop the incidence theory needed to
 prove both Theorem \ref{Incidence1} and Theorem \ref{Incidence3}.
  \begin{theorem}\label{incidence}Let $S_j$ be a non-degenerate quadratic
 surface in ${\Bbb F}_q^d$ defined as before.
 If $E$ is any subset of $S_j$, then we have
 \[ \sum_{\{(x,y)\in E\times E : x-y+z \in S_j\}}1 \lesssim   (\#E)^2
 q^{-1} +(\#E) q^{\frac{d-1}{2}}\]
 for all $z \in {\Bbb F}_q^d$ where the bound is independent of $z \in
 {\Bbb F}_q^d$.
 \end{theorem}
 \begin{proof} Fix $E \subset S_j.$ For each $z \in {\Bbb F}_q^d$, consider
 \begin{align*} &\sum_{\{(x,y)\in E\times E : x-y+z \in S_j\}}1\\
 &=\sum_{(x,y)\in {\Bbb F}_q^d \times {\Bbb F}_q^d} E(x)E(y) S_j(x-y+z)\\
 &=\sum_{(x,y)\in {\Bbb F}_q^d \times {\Bbb F}_q^d } E(x)E(y)\sum_{m\in
 {\Bbb F}_q^d} \chi(m\cdot(x-y+z))\widehat{S_j}(m)\\
 &=q^{2d} \sum_{m\in {\Bbb F}_q^d} |\widehat{E}(m)|^2\chi(m\cdot z)
 \widehat{S_j}(m)= \mbox{I} +\mbox{II}\end{align*}
 where
 \[\mbox{I}=q^{2d}|\widehat{E}(0,\cdots,0)|^2 \widehat{S_j}(0,\cdots,0)\]
 and
 \[\mbox{II} = q^{2d}\sum_{m\ne (0,\cdots,0)} |\widehat{E}(m)|^2\chi(m\cdot z)
 \widehat{S_j}(m).\]
 Using Lemma \ref{keylemma1} and  Plancherel , we obtain that
 \[\mbox{I} \approx (\#E)^2 q^{-1},\]
 and
 \begin{align*}|\mbox{II}| &\lesssim q^{2d} q^{-\frac{d+1}{2}} \sum_{m\ne
 (0,\cdots,0)} |\widehat{E}(m)|^2 \\
 &\le q^{2d}q^{-\frac{d+1}{2}}q^{-d} \sum_{x\in {\Bbb F}_q^d} |E(x)|^2
 =q^{\frac{d-1}{2}}(\#E).\end{align*}
 This completes the proof. \end{proof}
  \begin{corollary}\label{CorIncidence}
  Let $S_j$ be a non-degenerate quadratic surface in ${\Bbb F}_q^d$ and $E$
 be any subset of $S_j$. Then we have
 \[\sum_{\{(x,y,z,s)\in E^4 : x+z=y+s \}}1 \lesssim \min \{ (\#E)^3 ,
 (\#E)^3 q^{-1} +(\#E)^2 q^{\frac{d-1}{2}}\}.\]
 \end{corollary}
  \begin{proof}
 Since $ E$ is a subset of $S_j$, we have
 \[ \sum_{\{(x,y,z,s)\in E^4 : x+z=y+s \}}1 \leq \sum_{z\in E}
 \sum_{\{(x,y)\in E^2 : x-y+z \in S_j\}} 1.\]
 Thus Corollary \ref{CorIncidence} is the immediate result from Theorem
 \ref{incidence} and the obvious fact that
 \[\sum_{\{(x,y,z,s)\in E^4 : x+z=y+s \}}1 \leq (\#E)^3 .\] \end{proof}
\subsection{Proof of Theorem \ref{Incidence1}}
 In order to prove Theorem \ref{Incidence1}, We first expand the
 left-hand side of the inequality in (\ref{Bigsize}). It follows that
 \begin{align}  \|\widehat{E d\sigma}\|_{L^4({\Bbb F}_q^d,dm)} &=
 \Big(\sum_{m \in {\Bbb F}_q^d}
 |\widehat{E d\sigma}(m)|^4 \Big)^{\frac{1}{4}}\nonumber\\
 &= \Big(\sum_{m \in {\Bbb F}_q^d} \Big| \frac{1}{\#S_j} \sum_{x\in S_j}
 \chi(-x\cdot m) E(x)\Big|^4 \Big)^{\frac{1}{4}} \nonumber\\
 &= \frac{1}{\#S_j} \Big(\sum_{x,y,z,s \in E \subset S_j}  \sum_{m\in
 {\Bbb F}_q^d} \chi((x-y+z-s)\cdot m) \Big)^{\frac{1}{4}}\nonumber\\
 \label{Equality1}&= \frac{q^{\frac{d}{4}}}{\#S_j} \Big(
 \sum_{\{(x,y,z,s)\in E^4 : x+z=y+s \}}1     \Big)^{\frac{1}{4}}\end{align}
 Since  $q^{\frac{d+1}{2}} \lesssim \#E \lesssim q^{d-1}$ from the
 hypothesis,  we use Corollary \ref{CorIncidence} to obtain that
 \begin{equation}\label{Large}
 \sum_{\{(x,y,z,s)\in E^4 : x+z=y+s \}} 1 \lesssim (\#E)^3
 q^{-1}.\end{equation}
 Combining (\ref{Equality1}) with (\ref{Large}), we have
 \begin{equation}\label{Compare1}
  \|\widehat{E d\sigma}\|_{L^4({\Bbb F}_q^d,dm)}\lesssim
 \frac{q^{\frac{d-1}{4}} (\#E)^{\frac{3}{4}}} {\#S_j}.
 \end{equation}
 On the other hand, by expanding the right-hand side of the inequality in
 (\ref{Bigsize}), we see that
 \begin{equation} \label{Equality2} \|E\|_{L^{\frac{4}{3}} (S_j, d\sigma)}
 = \Big(\frac{\#E}{\#S_j}\Big)^{\frac{3}{4}}. \end{equation}
Since $\#S_j \approx q^{d-1}$ by Corollary \ref{cor1},  comparing
 (\ref{Compare1}) with (\ref{Equality2}) yields the inequality in
 (\ref{Bigsize}) and completes the proof.
\subsection{Proof of Theorem \ref{Incidence3}} In order to prove Theorem \ref{Incidence3},
we need the following lemma.
  \begin{lemma}\label{Incidence2}
 Let $S_j$ be a non-degenerate quadratic surface in ${\Bbb F}_q^d$ and $E$
 be a subset of $S_j.$
 For $p_0 \geq 2$ ,we have the following estimates
 \begin{equation}\label{First} \|\widehat{E d\sigma}\|_{L^4({\Bbb
 F}_q^d,dm)} \lesssim q^{\frac{-3d+5}{8} +\frac{d-1}{2p_0}}
 \|E\|_{L^{p_0}(S_j, d\sigma)} \quad \mbox{for} \quad
 1 \lesssim \#E \lesssim q^{\frac{d-1}{2}}\end{equation}
 \begin{equation}\label{Second}\|\widehat{E d\sigma}\|_{L^4({\Bbb
 F}_q^d,dm)} \lesssim q^{\frac{-3d+9}{8} +\frac{d-3}{2p_0}}
 \|E\|_{L^{p_0}(S_j, d\sigma)} \quad \mbox{for}\quad
 q^{\frac{d-1}{2}}\lesssim \#E \lesssim q^{\frac{d+1}{2}},\end{equation}
 and
 \begin{equation}\label{Third}\|\widehat{E d\sigma}\|_{L^4({\Bbb
 F}_q^d,dm)} \lesssim q^{\frac{-3d+9}{8} +\frac{d-3}{2p_0}}
 \|E\|_{L^{p_0}(S_j, d\sigma)} \quad
 \mbox{for} \quad  1 \lesssim \#E \lesssim q^{\frac{d+1}{2}}.\end{equation}
 \end{lemma}

 \begin{proof}
 We first prove the inequality in (\ref{First}). Since $1 \lesssim \#E
 \lesssim q^{\frac{d-1}{2}}$, using Corollary \ref{CorIncidence} we have
 \[\sum_{\{(x,y,z,s)\in E^4 : x+z=y+s \}}1 \lesssim  (\#E)^3 .\]  Combining
 this with the fact in (\ref{Equality1}) , we obtain that
 \begin{equation}\label{Tag11}\|\widehat{E d\sigma}\|_{L^4({\Bbb
 F}_q^d,dm})\lesssim  \frac{q^{\frac{d}{4}} (\#E)^{\frac{3}{4}}}
 {\#S_j}.\end{equation}
 As before, we note that
 \begin{equation}\label{Tag12} \|E\|_{L^{p_0}(S_j, d\sigma)} \approx \Big(
 \frac{\#E}{\#S_j} \Big)^{\frac{1}{p_0}}.\end{equation}
 From (\ref{Tag11}) and (\ref{Tag12}), it suffices to show that for every $
 1\lesssim \#E \lesssim q^{\frac{d-1}{2}},$
 \begin{equation}\label{Tag13} \frac{q^{\frac{d}{4}}
 (\#E)^{\frac{3}{4}-\frac{1}{p_0} }} {(\#S_j)^{1-\frac{1}{p_0}}} \lesssim
 q^{\frac{-3d+5}{8}+\frac{d-1}{2p_0} } \end{equation}
 Since $p_0 \geq2$ and $\#S_j \approx q^{d-1} ,$ the inequality in
 (\ref{Tag13}) follows by a direct calculation.
 Thus the inequality in (\ref{First}) holds. In order to prove the
 inequality in (\ref{Second}), just note from Corollary \ref{CorIncidence}
 that  since $ q^{\frac{d-1}{2}} \lesssim \#E \lesssim q^{\frac{d+1}{2}}$, we
 have
 \[\sum_{\{(x,y,z,s)\in E^4 : x+z=y+s \}}1 \lesssim  (\#E)^2
 q^{\frac{d-1}{2}},\]
 and then follow the same argument as in the proof of the inequality
 (\ref{First}).
 The inequality in (\ref{Third}) follows from the inequalities in
 (\ref{First}) and (\ref{Second}) because
 \[ q^{\frac{-3d+5}{8}+\frac{d-1}{2p_0} }  \lesssim
 q^{\frac{-3d+9}{8}+\frac{d-3}{2p_0} } , \quad \mbox{for} \quad p_0 \geq2.  \]
 Thus the proof of Lemma \ref{Incidence2} is complete. \end{proof}
We now return to the proof of Theorem \ref{Incidence3}.
 From (\ref{Tag14}), recall that we have
 \begin{equation}\label{Tag15} \|\widehat{Ed\sigma} \|_{L^2({\Bbb
 F}_q^d,dm)}\approx q^{\frac{1}{2}} \|E\|_{L^2(S_j,d\sigma)} \end{equation}
 for all characteristic functions $E(x)$ on $S_j.$
 Therefore  Theorem \ref{Incidence3} can be obtained by interpolating
 (\ref{Tag15}) and the inequalities in Lemma \ref{Incidence2}.

\end{document}